\documentclass[11pt]{article}
\usepackage{amsfonts,amsmath,tikz}
\usepackage{epsfig}
\usepackage{latexsym}
\usepackage{color}
\usepackage{epsf,amssymb}

\usepackage{amsthm}
\usepackage{yfonts}
\usepackage{hyperref}
\usepackage{ifthen}

\newtheorem{theorem}{Theorem}
\newtheorem{lemma}[theorem]{Lemma}
\newtheorem{observation}[theorem]{Observation}

\def\vertex(#1){\put(#1){\circle*{2}}}
\def\vertexo(#1){\put(#1){\circle{2}}}
\def\vert(#1){\put(#1){\circle*{1.5}}}
\def\verto(#1){\put(#1){\circle{1.5}}}
\def\lab(#1)#2{\put(#1){\makebox(0,0)[c]{#2}}}
\newcommand{\edim}{{\rm edim}}

\tikzset{My Style/.style={draw, circle, fill=black,scale=0.3}} 

\title{Graphs with the edge metric dimension smaller than the metric dimension
\footnotetext{\small\tt knor@math.sk, smajstor@mathos.hr, aodenteo@gmail.com, skrekovski@gmail.com, and\\ \hspace*{0.55cm}ismael.gonzalez@uca.es}}

\author{Martin Knor$^1$, Snje\v{z}ana Majstorovi\'c$^2$, Aoden Teo Masa Toshi$^3$, \\Riste \v Skrekovski$^{4,5}$ and Ismael G. Yero$^6$\\[0.3cm]
\small $^1$ \it Slovak University of Technology in Bratislava, Bratislava, Slovakia \\[0.1cm]
\small $^2$ \it University of Osijek, Department of Mathematics, Croatia\\[0.1cm]
\small $^3$ \it Independent researcher, Singapore\\[0.1cm]
\small $^4$ \it University of Ljubljana, FMF, 1000 Ljubljana, Slovenia \\[0.1cm]
\small $^5$ \it Faculty of Information Studies, 8000 Novo Mesto, Slovenia \\[0.1cm]
\small $^6$ \it Universidad de C\'adiz, Departamento de Matem\'aticas, Algeciras, Spain\\
}

\begin{document}
\maketitle
\begin{abstract}
Given a connected graph $G$, the metric (resp. edge metric) dimension of
$G$ is the cardinality of the smallest ordered set of vertices that uniquely
identifies every pair of distinct vertices (resp. edges) of $G$ by means
of distance vectors to such a set. In this work, we settle three open problems on (edge) metric dimension of graphs.
Specifically, we show that for every $r,t\ge 2$ with $r\ne t$, there is $n_0$, such that for every $n\ge n_0$ there exists a graph $G$ of order $n$ with metric dimension $r$ and edge metric dimension $t$, which among other consequences, shows the existence of infinitely many graph whose edge metric dimension is strictly smaller than its metric dimension. In addition, we also prove that it is not possible to bound the edge metric dimension of a graph $G$ by some constant factor of the metric dimension of $G$.
\end{abstract}







\section{Introduction}

Metric dimension is nowadays a well studied topic in graph theory and
combinatorics, as well as in some computer science applications, and the theory involving it is indeed full of interesting
results and open questions.
One recent issue that has attracted the attention of several researchers
concerns a variant of the standard metric dimension, in which it is required to uniquely
recognize the edges of a graph, instead of its vertices, and by using vertices as the recognizing elements.
This variant was introduced in \cite{Kel}, and since its appearance, a
significant number of works have been published.
In this sense, we mention the most recent ones
\cite{Filipovi2019,Geneson2020,Peterin2020,Zhang2020,Zhu,Zubrilina}. Specifically, the edge metric dimension is studied in several situations as follows: \cite{Filipovi2019} is dedicated to study several generalized Petersen graphs; in \cite{Geneson2020}, a number of results about pattern avoidance in graphs with bounded edge metric dimension are given; \cite{Peterin2020} centers the attention on some product graphs (corona, join and lexicographic); \cite{Zhang2020} studies some convex polytopes and related graphs; in \cite{Zhu}, a characterization of graphs with the largest possible edge metric dimension (order minus one) is given; and finally, in \cite{Zubrilina} the Cartesian product of any graph with a path is studied, as well as, it is proved to be not possible to bound the metric dimension of a graph $G$ by some constant factor of the edge metric dimension of $G$. We should also remark some results from the seminal article \cite{Kel}. There was proved for instance that computing the edge metric dimension of graphs is NP-hard, that can be approximated within a constant factor, and that becomes polynomial for the case of trees. Further, some bounds and closed formulas for several classes of graphs including trees, grid graphs and wheels (among others), were also deduced in \cite{Kel}.

We recall that the parameter edge metric dimension (from \cite{Kel}) studied here is not the same as that one defined in \cite{Nasir}, where the authors studied the metric dimension of the line graph of a graph (namely edges uniquely recognizing edges), and called such parameter as edge metric dimension, although it was further renamed as the edge version of metric dimension in \cite{Liu}.

In the next we recall the necessary terminology and notation.
We consider only simple and connected graphs.
Let $G$ be a graph and let $u,v$ be its vertices.
By $d_G(u,v)$ (or by $d(u,v)$ when no confusion is likely) we denote the
distance from $u$ to $v$ in $G$.
Let $z$ be a vertex of $G$.
We say that $z$ \emph{identifies} (\emph{resolves} or \emph{determines})
a pair of vertices $u,v\in V(G)$, if $d_G(u,z)\ne d_G(v,z)$,
An ordered set of vertices $S$ is a \emph{metric generator} for $G$
if every two vertices $u,v\in V(G)$ are identified by a vertex of $S$.
The \emph{metric dimension} of $G$ is then the cardinality
of the smallest metric generator for $G$.
Such cardinality is denoted by $\dim(G)$ and a metric generator
of cardinality $\dim(G)$ is known as a \emph{metric basis}.
It is necessary to remark that the concepts above were first defined
in \cite{Blumenthal1953} for a more general setting of metric spaces.
The concepts were again independently rediscovered for the case
of graphs in \cite{Harary1976} and \cite{Slater1975}, where metric
generators were called resolving sets and locating sets, respectively.
Also, in \cite{Slater1975}, the metric dimension was called locating
number.
The terminology of metric generators was first used in \cite{Sebo2004}.

Let $G$ be a graph and let $S$ be an ordered set of vertices of $G$.
For every $v\in V(G)$ we can consider the vector $r(v|S)$ of distances
from $v$ to the vertices in $S$.
If $S$ is a metric generator, then all such vectors are pairwise different.
The vector $r(v|S)$ is known as the \emph{metric representation} of $v$
with respect to $S$.

The concept of edge metric dimension was first described in \cite{Kel},
as a way to uniquely recognize the edges of a given graph $G$.
A vertex $z\in V(G)$ \emph{distinguishes} two edges $e,f\in E(G)$
if $d_G(e,z)\ne d_G(f,z)$, where $d_G(e,z)=d_G(uv,z)=\min\{d_G(u,z),d_G(v,z)\}$.
A set of vertices $S\subset V(G)$ is an \emph{edge metric generator}
for $G$, if any two edges of $G$ are distinguished by a vertex of $S$.
The \emph{edge metric dimension} of $G$ is the cardinality of the smallest
edge metric generator for $G$, and is denoted by $\edim(G)$.
An edge metric generator of cardinality $\edim(G)$ is known as an
\emph{edge metric basis}.
The edge metric representation is defined analogously as in the case
of the metric dimension.


\section{Edge metric dimension versus metric dimension}

One would expect that $\dim(G)$ and $\edim(G)$ are related.
The search for a relationship between these two invariants (in a shape
of a bound for instance) was of interest in the seminal article \cite{Kel},
as well as in the subsequent works on the topic (see also for instance \cite{Zubrilina}).
In this sense, in \cite{Kel}, several families of graphs for which
$\dim(G)<\edim(G)$, or $\dim(G)=\edim(G)$, or $\dim(G)>\edim(G)$ were
presented.
For the last case, only one construction was given, namely the Cartesian
product of two cycles $C_{4r}\Box C_{4t}$.
It was shown in \cite{Kel} that
$\edim(C_{4r}\Box C_{4t})=3<4=\dim(C_{4r}\Box C_{4t})$. In consequence, it was claimed in \cite{Kel} that the metric dimension and the edge metric dimension of graphs seemed to be not comparable in general. This example above and the other results from \cite{Kel} allowed the authors of that article to point out the following questions.
\begin{itemize}
\item[{\rm (i)}]
For which integers $r,t,n\ge 1$, with $r,t\le n-1$, can be constructed a graph $G$ of order $n$ with $\dim(G)=r$ and $\edim(G)=t$?
\item[{\rm (ii)}]
Is it possible that $\dim(G)$ or $\edim(G)$ would be bounded from above
by some constant factor of $\edim(G)$ or $\dim(G)$, respectively?
\item[{\rm (iii)}] Can you construct some other families of graphs
for which $\dim(G)>\edim(G)$?
\end{itemize}

Note that the question (i) is precisely the realization of graphs $G$ with prescribed values on its order, metric dimension and edge metric dimension, while the question (ii) is equivalent to ask whether the ratios
$\frac{\edim(G)}{\dim(G)}$ and $\frac{\dim(G)}{\edim(G)}$ are bounded
from above. The third question can be settled as a consequence of the other two. Realization results concerning the case in which $\dim(G)\le \edim(G)$ were already studied in \cite{Kel}, although not completed. With respect to the ratios, it was proved in
\cite{Zubrilina} that $\frac{\edim(G)}{\dim(G)}$ is not bounded from above. The other possibility has never been studied till now.

In this work we deal with these three problems mentioned above. That is, our results complete the unboundedness results given in \cite{Zubrilina}, while studying the ratio $\frac{\dim(G)}{\edim(G)}$, and thus, the problem in (ii) is now completely settled. We also give positive answer to (iii), and moreover, we present an almost complete answer to (i), since we show that for every $r,t\ge 2$ with $r\ne t$, there is $n_0$, such that for every $n\ge n_0$ there exists an outerplanar graph $G$ (a cactus graph indeed), of order $n$ with $\dim(G)=r$ and $\edim(G)=t$. This result is in a sense best possible, because if $\edim(G)=1$, then $G$ is a path of length at least $2$, and consequently
$\dim(G)=1$ as well. We remark that a similar result for $2\le r\le t\le 2r$ was proved in \cite{Kel},
where $n_0$ was shown to be at most $2r+2$. So, our result complements the former one and a weaker version of this former
result (with a weaker bound for $n_0$) can be proved also using a variant of
our construction.

As a consequence of our results it is clear the existence of infinite families of graphs $G$ for which
the differences $\edim(G)-\dim(G)$ and $\dim(G)-\edim(G)$ are arbitrarily large. Proving that the difference $\edim(G)-\dim(G)$ is arbitrarily large was already presented in \cite{Kel}. However, the other difference $\dim(G)-\edim(G)$ was only proved to be at most 1 in \cite{Kel}, and there was no more knowledge on this issue. Clearly, the unboundedness of the ratio $\frac{\dim(G)}{\edim(G)}$ gives as a consequence that $\dim(G)-\edim(G)$ can as large as possible.

While graphs for which $\dim(G)<\edim(G)$ are very common, and they are present in several investigations already published, the opposed version $\dim(G)<\edim(G)$ seemed to be more elusive till now. We have first observed that $K_2$ is the unique connected simple graph whose edge metric
dimension is $0$. Since $\dim(K_2)=1$, $K_2$ is the smallest graph which has the edge metric
dimension smaller than the metric dimension.
For non-trivial examples one needs to consider graphs of order at least 10.
By exhaustive computer search we found that the smallest possible graphs $G$
(different from $K_2$) satisfying the inequality $\edim(G)<\dim(G)$, are the
five graphs on $10$ vertices depicted in Figure~{\ref{fig:g10}}. Moreover, we have found 61 such graphs of order 11.

\begin{figure}[ht]
\centering
\includegraphics[width=130mm, height=80mm]{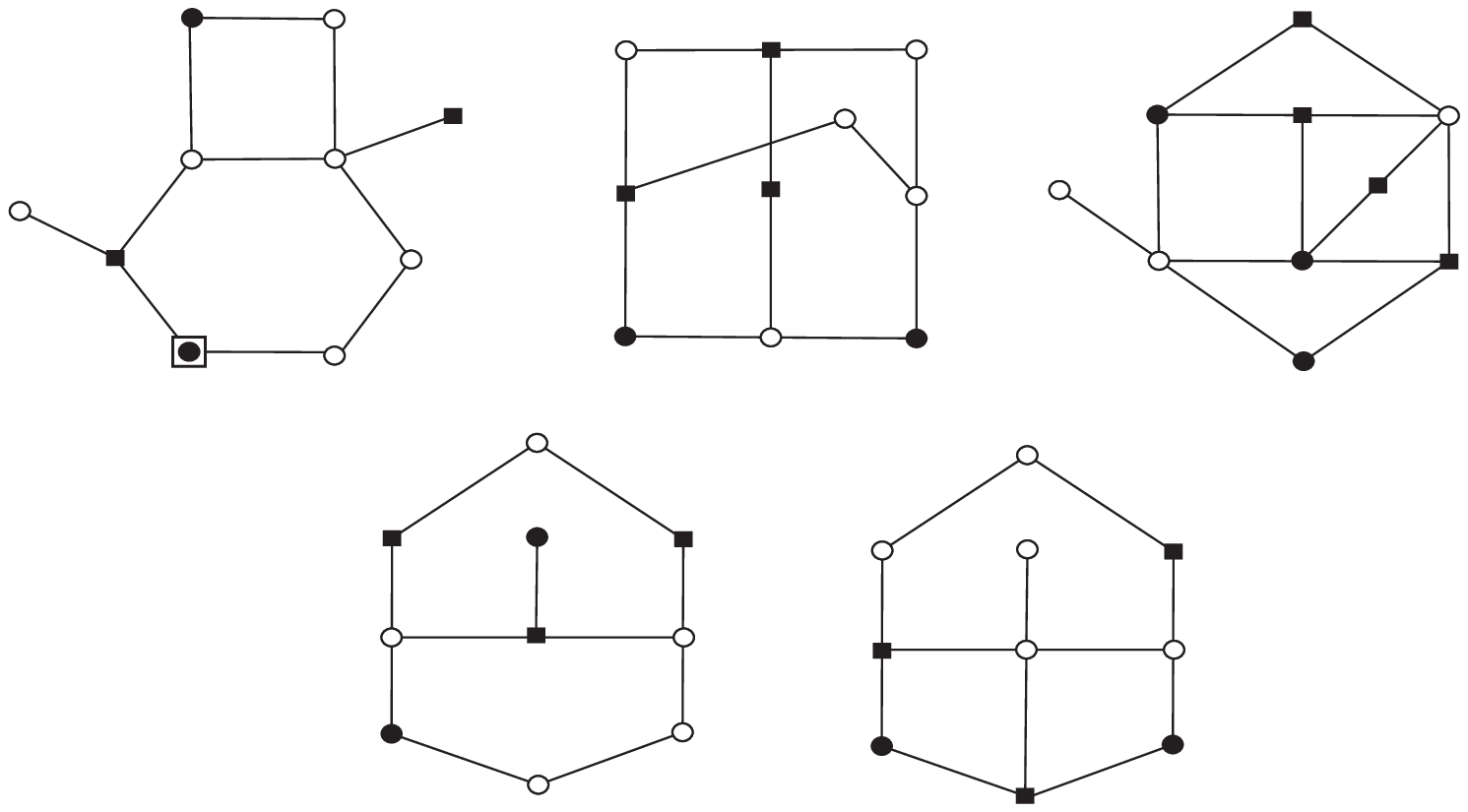}
\caption{The smallest graphs $G$ for which  $\edim(G)<\dim(G)$, different from $K_2$. The squared vertices form a metric basis and the circled bolded vertices form an edge metric basis.}
\label{fig:g10}
\end{figure}

The main contributions of our work are as follows.

\begin{theorem}
\label{th:main}
Let $k_1,k_2\ge 2$ and $k_1\ne k_2$. Then there is an integer $n_0$ such that for every $n\ge n_0$ there exists a graph on
$n$ vertices with $\dim(G)=k_1$ and $\edim(G)=k_2$.
\end{theorem}

\begin{theorem}
\label{th:main-2}
The ratio $\frac{\dim(G)}{\edim(G)}$ is not bounded from above.
\end{theorem}

\begin{proof}
By Theorem \ref{th:main}, for arbitrarily large $N$, it is always possible to find a graph $G$ such that $\dim(G) = Nk$ and $\edim(G)= k$. As such, $\frac{\dim(G)}{\edim{G}}$ can be made arbitrarily large.
\end{proof}

Notice that by using a similar argument as the one in the proof above, although it is already known from \cite{Zubrilina}, we can also prove that the ratio $\frac{\edim(G)}{\dim(G)}$ is not bounded from above.

\section{Proof of Theorem \ref{th:main}}

In order prove the main result of this work, we shall construct infinite families of graphs $G$ for which $\edim(G)<\dim(G)$ as well as other ones where $\dim(G)<\edim(G)$. To this end, we need first some preliminary results. We first describe two graphs that will be used in this purpose.
Take a cycle $C$ on $n_1$ vertices, where $n_1\ge 5$.
We denote the vertices of $C$ consecutively by $a_1,a_2,\dots, a_{n_1}$.
Further, take a path $P$ on $n_2$ vertices denoted consecutively by
$b_1,b_2,\dots,b_{n_2}$, where $n_2\ge 1$, and join $P$ to $C$ by the edge
$a_2b_1$.
Then take vertices $c$ and $i$ and connect them by edges to $a_{n_1}$ and
$a_1$, respectively.
Finally, take $n_3$ vertices $j_1,j_2,\dots,j_{n_3}$, where $n_3\ge 2$,
and join them by edges to the vertex $i$.
We denote the resulting graph by $G_{n_1,n_2,n_3}$. A fairly representative example of a graph as described above in drawn in Figure \ref{fig:G-7-3-2}.

\begin{figure}[h]
\centering
\includegraphics[width=40mm, height=45mm]{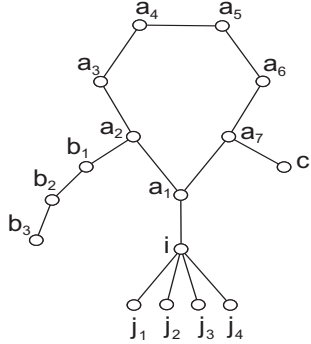}
\caption{The graph $G_{7,3,4}$}
\label{fig:G-7-3-2}
\end{figure}

In connection with the graphs $G_{n_1,n_2,n_3}$, we shall need a graph denoted by $G_{n_1,n_2}$ that is obtained from $G_{n_1,n_2,n_3}$ by removing the vertices of the set $\{i,j_1,j_2,\dots,j_{n_3}\}$ and the edges incident with them. Thus, $G_{n_1,n_2}$ is a proper subgraph of $G_{n_1,n_2,n_3}$.

\subsection{Preliminaries of the proof}

We now prove several auxiliary results about the graph $G_{n_1,n_2,n_3}$. We remark that we shall be using the graphs $G_{n_1,n_2}$ in most of the described situations.

\begin{observation}
\label{obs:lower}
Let $n_1\ge 5$, $n_2\ge 1$ and $n_3\ge 2$.
Then $\dim(G_{n_1,n_2,n_3})\ge n_3$ and $\edim(G_{n_1,n_2,n_3})\ge n_3$.
\end{observation}

\begin{proof}
Let $S$ be a metric basis of $G_{n_1,n_2,n_3}$.
If $S$ does not contain a vertex of the subgraph $G_{n_1,n_2}$, then $r(a_2|S)=r(a_{n_1}|S)$.
Thus, $S$ contains at least one vertex of $G_{n_1,n_2}$.
Moreover, if $S$ does not contain two pendant vertices $j_r$ and $j_t$
attached to $i$, then $r(j_r|S)=r(j_t|S)$.
Thus, $S$ contains at least $n_3-1$ pendant vertices attached to $i$, and so, in conclusion we
get $\dim(G_{n_1,n_2,n_3})\ge n_3$.

Now let $T$ be an edge metric basis. Here the situation is analogous since
$r(a_1a_2|T)=r(a_1a_{n_1}|,T)$ if $T$ does not contain a vertex of
$G_{n_1,n_2}$, and $r(j_ri|T)=r(j_ti|T)$ if $j_r$ and $j_t$ are pendant
vertices attached to $i$ and not in $T$.
\end{proof}

\begin{lemma}
\label{lem:upper}
Let $n_1\ge 5$, $n_2\ge 1$ and $n_3\ge 2$.
Then $\dim(G_{n_1,n_2,n_3})\le n_3+1$ and $\edim(G_{n_1,n_2,n_3})\le n_3+1$.
\end{lemma}

\begin{proof}
Let $S=\{j_1,j_2,\dots,j_{n_3-1},a_{\alpha},a_{\beta}\}$, where
$\alpha=\lfloor\frac{n_1+1}2\rfloor$ and
$\beta=\lceil\frac{n_1+3}2\rceil$.
Observe that $d(a_1,a_{\alpha})=d(a_1,a_{\beta})$ since
$d(a_1,a_{\alpha})=\lfloor\frac{n_1+1}2\rfloor-1$ and
$d(a_1,a_{\beta})=n_1+1-\lceil\frac{n_1+3}2\rceil$.
Moreover, $d(a_{\alpha},a_{\beta})=1$ if $n_1$ is odd and
$d(a_{\alpha},a_{\beta})=2$ otherwise.
Obviously,  $S$ contains $n_3+1$ vertices.
And since $n_1\ge 5$, we have $\beta<n_1$.
Denote $\gamma=d(a_{\alpha},a_1)=\lfloor\frac{n_1-1}2\rfloor$.

First we show that $S$ is a metric generator of $G_{n_1,n_2,n_3}$.
Since $n_3\ge 2$, there is a vertex of $S$ outside the subgraph $G_{n_1,n_2}$.
This vertex distinguishes those vertices of $G_{n_1,n_2}$ which are at different distances
from $a_1$, but not those which are at the same distance from $a_1$.
According to the distance from $a_1$, the vertices of $G_{n_1,n_2}$ are partitioned into
nontrivial sets $\{a_2,a_{n_1}\}$, $\{a_3,a_{n_1-1},b_1,c\}$,
$\{a_4,a_{n_1-2},b_2\}$, etc.
Some of these sets are probably smaller since they do not need to contain
vertices of both $C$ and $P$, and $n_1$ could be even.

Let $\delta=\lfloor\frac{n_1}2\rfloor$.
For every $v\in V(G_{n_1,n_2})$, denote
$\tilde{r}(v)=(d(v,a_{\alpha}),d(v,a_{\beta}))$.
Then the vertices in the first set of the partition have
$\tilde{r}(a_2)=(\gamma-1,\delta)$
and $\tilde{r}(a_{n_1})=(\delta,\gamma-1)$.
Since $\gamma\le\delta$, these pairs are different.
Further, the pairs $\tilde{r}$ for vertices in the second set of the
partition are
$(\gamma-2,\delta-1)$, $(\delta-1,\gamma-2)$, $(\gamma,\delta+1)$,
$(\delta+1,\gamma)$, and they are pairwise distinct too.
Since the distances to vertices of $C$ decrease while the distances to
vertices of $P$ increase, it is obvious that the pair $a_{\alpha},a_{\beta}$
distinguishes the vertices in the sets of the partition.
Hence, $S$ identifies the vertices of $G_{n_1,n_2}$.

Since a pendant vertex in $S$ is identified trivially, it remains to check
the vertices $i$ and $j_{n_3}$.
Obviously, $j_{n_3}$ is the only vertex at distance $2$ from the pendant
vertices in $S$ and at distance $\gamma+2$ from $a_{\alpha}$.
On the other hand, $i$ is the only vertex at distance $1$ from the pendant
vertices in $S$.
Thus, $S$ is a metric generator of $G_{n_1,n_2,n_3}$.

Now we show that $S$ is an edge metric generator of $G_{n_1,n_2,n_3}$.
We proceed analogously as in the case for the metric generator.
According to the distance from $a_1$, the vertices of $S$ outside $G_{n_1,n_2}$ partition
the edges of $G$ into sets $\{a_1a_2,a_1a_{n_1}\}$,
$\{a_2a_3,a_{n_1}a_{n_1-1},a_2b_1,a_{n_1}c\}$,
$\{a_3a_4,a_{n_1-1}a_{n_1-2},b_1b_2\}$, etc.
Again, some of these sets are probably smaller since they do not need to contain the
edges of both $C$ and $P$, and $n_1$ may be odd.
For every $e\in E(G_{n_1,n_2,n_3})$, denote $\tilde{r}(v)=(d(e,a_{\alpha}),d(e,a_{\beta}))$.
Then the pairs $\tilde{r}$ for edges in the first set of the partition are
$(\gamma-1,\gamma)$ and $(\gamma,\gamma-1)$.
Further, the pairs $\tilde{r}$ for edges in the second set of the
partition are $(\gamma-2,\delta-1)$, $(\delta-1,\gamma-2)$, $(\gamma-1,\delta)$,
$(\delta,\gamma-1)$ and they are pairwise distinct too.
Since the distances to edges of $C$ decrease while the distances to edges
of $P$ increase, it is obvious that the pair $a_{\alpha},a_{\beta}$
distinguishes the edges in the sets of the partition.
Hence, $S$ identifies the edges of $G_{n_1,n_2}$.

Since a pendant edge $ij_t$ is identified trivially if $j_t\in S$, it
remains to check the edges $ia_0$ and $ij_{n_3}$.
Obviously, $ij_{n_3}$ is the only edge at distance $1$ from the pendant vertices
in $S$ and at distance $\gamma+1$ from $a_{\alpha}$.
On the other hand, $ia_1$ is the only edge at distance $1$ from the pendant
vertices in $S$ and at distance $\gamma$ from $a_{\alpha}$.
Therefore, $S$ is a metric generator of $G_{n_1,n_2,n_3}$ and the proof is completed.
\end{proof}

The next two propositions show that $\dim(G_{n_1,n_2,n_3})$ and $\edim(G_{n_1,n_2,n_3})$ depend on the parity of $n_1$.

\begin{lemma}
\label{prop:vertex}
Let $n_1\ge 5$, $n_2\ge 1$ and $n_3\ge 2$.
Then $\dim(G_{n_1,n_2,n_3})=n_3$ if $n_1$ is odd, and
$\dim(G_{n_1,n_2,n_3})=n_3+1$ if $n_1$ is even.
\end{lemma}

\begin{proof}
First assume that $n_1$ is odd.
Analogously as in the proof of Lemma~{\ref{lem:upper}}, let
$\alpha=\lfloor\frac{n_1+1}2\rfloor$, $\gamma=\lfloor\frac{n_1-1}2\rfloor$
and $\delta=\lfloor\frac{n_1}2\rfloor$.
Since $n_1$ is odd, $\gamma=\delta$.
As shown in Observation~{\ref{obs:lower}}, every metric basis of $G_{n_1,n_2,n_3}$ contains a
vertex outside $G_{n_1,n_2}$.
By using the distances, this vertex partitions $V(G_{n_1,n_2})$ into nontrivial sets
$P_1=\{a_2,a_{n_1}\}$, $P_2=\{a_3,a_{n_1-1},b_1,c\}$,
$P_3\subseteq\{a_4,a_{n_1-2},b_2\}$, etc.
For every $v\in V(G_{n_1,n_2})$, we denote $\tilde{r}(v)=d(a_{\alpha},v)$.
Then the values of $\tilde{r}$ for the vertices in $P_1$ are $\gamma-1$ and
$\delta$, and they are different.
The values of $\tilde{r}$ for vertices in $P_2$ are $\gamma-2$, $\delta-1$,
$\gamma$ and $\delta+1$, and they are different too.
The set $P_3$ contains vertices with values of $\tilde{r}$ being $\gamma-3$,
$\delta-2$ and $\gamma+1$, etc.
Hence, to identify the vertices of $G_{n_1,n_2}$ it suffices (and is necessary by
the proof of Observation~{\ref{obs:lower}}) that a metric generator of $G_{n_1,n_2,n_3}$ will contain
only one vertex of $G_{n_1,n_2}$, namely $a_{\alpha}$.
Hence, $\{j_1,j-2,\dots,j_{n_3-1},a_{\alpha}\}$ is a metric generator of $G_{n_1,n_2,n_3}$,
see the proof of Lemma~{\ref{lem:upper}}.
By Observation~{\ref{obs:lower}}, we then get $\dim(G_{n_1,n_2,n_3})=n_3$.

We now assume $n_1$ is even and proceed analogously as above. Vertices outside $G_{n_1,n_2}$ partition $V(G_{n_1,n_2,n_3})$ into nontrivial
sets $P_1=\{a_2,a_{n_1}\}$, $P_2=\{a_3,a_{n_1-1},b_1,c\}$,
$P_3\subseteq\{a_4,a_{n_1-2},b_2\}$, etc.
We denote this partition by $\mathcal P$.
(Though it is not obvious, this partition is slightly different from the
partition for the case when $n_1$ is odd.)
We show that there is not a unique vertex in $G_{n_1,n_2}$ which distinguishes all the vertices inside
the sets of $\mathcal P$.
Let $v\in V(G_{n_1,n_2,n_3})$. By symmetry, it suffices to distinguish four cases:

\noindent
{\bf Case 1:} {\it $v\in\{a_1,a_{\frac{n_1+2}2}\}$.} Then $d(v,a_2)=d(v,a_{n_1})$ and $a_2,a_{n_1}\in P_1$.

\noindent
{\bf Case 2:} {\it $v\in\{a_2,b_1,b_2,\dots,b_{n_2}\}$.} Then $d(v,c)=d(v,a_{n_1-1})$ and $a_{n_1-1},c\in P_2$.

\noindent
{\bf Case 3:} {\it $v\in\{a_3,a_4,\dots,a_{\frac{n_1-2}2}\}$.} (This set is empty if $n_1=6$.)
Then again $d(v,c)=d(v,a_{n_1-1})$.

\noindent
{\bf Case 4:} {\it $v=a_{\frac{n_1}2}$.} Then $d(v,b_1)=d(v,a_{n_1-1})$ and $a_{n_1-1},b_1\in P_2$.

Since the other possibilities are symmetric, every metric basis of $G_{n_1,n_2,n_3}$ contains at least two vertices of $G_{n_1,n_2}$.
And since every metric basis of $G_{n_1,n_2,n_3}$ contains at least $n_3-1$ pendant vertices attached to $i$, we have $\dim(G_{n_1,n_2,n_3})\ge n_3+1$. Thus, by Lemma~{\ref{lem:upper}}, $\dim(G_{n_1,n_2,n_3})=n_3+1$.
\end{proof}

We now consider the counterpart of Lemma \ref{prop:vertex} for the edge metric dimension case.

\begin{lemma}
\label{prop:edge}
Let $n_1\ge 5$, $n_2\ge 1$ and $n_3\ge 2$.
Then $\edim(G_{n_1,n_2,n_3})=n_3+1$ if $n_1$ is odd, and
$\edim(G_{n_1,n_2,n_3})=n_3$ if $n_1$ is even.
\end{lemma}

\begin{proof}
First assume that $n_1$ is odd. Since $n_3\ge 2$, every edge metric basis contains a vertex outside of $G_{n_1,n_2}$, and by
using distances, this vertex partitions $E(G_{n_1,n_2,n_3})$ into sets $P_1=\{a_1a_2,a_1a_{n_1}\}$, $P_2=\{a_2a_3,a_{n_1}a_{n_1-1},a_2b_1,a_{n_1}c\}$, $P_3\subseteq\{a_3a_4,a_{n_1-1}a_{n_1-2},b_1b_2\}$, etc.
We show that there is no vertex in $G_{n_1,n_2,n_3}$ which distinguishes edges inside
these sets. Let $v\in V(G_{n_1,n_2,n_3})$. By symmetry, it suffices to distinguish the following four situations.

\noindent
{\bf Case 1:} {\it $v=a_1$.} Then $d(v,a_1a_2)=d(v,a_1a_{n_1})$ and $a_1a_2,a_1a_{n_1}\in P_1$.

\noindent
{\bf Case 2:} {\it $v\in\{a_2,b_1,b_2,\dots,b_{n_2}\}$.}
Then $d(v,a_{n_1}a_{n_1-1})=d(v,a_{n_1}c)$ and $a_{n_1}a_{n_1-1},a_{n_1}c\in P_2$.

\noindent
{\bf Case 3:} {\it $v\in\{a_3,a_4,\dots,a_{\frac{n_1-1}2}\}$.}
(This set is empty if $n_1=5$.)
Then again $d(v,a_{n_1}a_{n_1-1})=d(v,a_{n_1}c)$.

\noindent
{\bf Case 4:} {\it $v=a_{\frac{n_1+1}2}$.}
Then $d(v,a_{n_1}a_{n_1-1})=d(v,a_2b_1)$ and $a_{n_1}a_{n_1-1},a_2b_1\in P_2$.

Since the other possibilities are symmetric, every edge metric basis of
$G_{n_1,n_2,n_3}$ contains at least two vertices of $G_{n_1,n_2}$.
And since also every edge metric basis of $G_{n_1,n_2,n_3}$ contains at least $n_3-1$ pendant
vertices attached to $i$, we have $\edim(G_{n_1,n_2,n_3})\ge n_3+1$.
By Lemma~{\ref{lem:upper}}, we then have $\edim(G_{n_1,n_2,n_3})=n_3+1$.

Now assume that $n_1$ is even. Analogously to the proof of Lemma~{\ref{lem:upper}}, let
$\alpha=\lfloor\frac{n_1+1}2\rfloor$, $\gamma=\lfloor\frac{n_1-1}2\rfloor$
and $\delta=\lfloor\frac{n_1}2\rfloor$.
Since $n_1$ is even, $\gamma=\delta-1$.
The vertices of an edge metric basis outside $G_{n_1,n_2}$ partition the edges of $G_{n_1,n_2,n_3}$
into sets $P_1=\{a_1a_2,a_1a_{n_1}\}$, $P_2=\{a_2a_3,a_{n_1}a_{n_1-1},a_2b_1,a_{n_1}c\}$,
$P_3\subseteq\{a_3a_4,a_{n_1-1}a_{n_1-2},b_1b_2\}$, etc.
For every $e\in E(G_{n_1,n_2,n_3})$, let us denote $\tilde{r}(e)=d(a_{\alpha},e)$.
Then the values of $\tilde{r}$ for the edges in $P_1$ are $\gamma-1$ and
$\gamma$ and they are different.
The values of $\tilde{r}$ for edges in $P_2$ are
$\gamma-2$, $\delta-1$, $\gamma-1$ and $\delta$ and they are different too.
The values of $\tilde{r}$ for edges in $P_3$ are $\gamma-3$, $\delta-2$ and $\gamma$, etc.
Hence, in order to identify the edges of $G_{n_1,n_2,n_3}$, it suffices (and it is indeed necessary by
the proof of Observation~{\ref{obs:lower}}) that an edge metric generator will contain only one vertex of $G_{n_1,n_2}$, namely $a_{\alpha}$.
Hence, the set $\{j_1,j_2,\dots,j_{n_3-1},a_{\alpha}\}$ is an edge metric generator of $G_{n_1,n_2,n_3}$, see the proof of Lemma~{\ref{lem:upper}}. Therefore, by Observation~{\ref{obs:lower}}, we get $\edim(G_{n_1,n_2,n_3})=n_3$.
\end{proof}

\subsection{Core of the proof}

To obtain the required graphs we are searching for, we will connect several copies of the graph $G_{n_1,n_2,n_3}$ by
adding a few edges. To this end, we need the following powerful tool.

\begin{lemma}
\label{lem:operation}
Let $G_1$ and $G_2$ be two graphs which are not paths, such that for any $i\in \{1,2\}$,
the graph $G_i$ contains a metric basis $S_i$ and an edge metric basis $T_i$ satisfying
the following conditions.
\begin{itemize}
\item[{\rm (1)}]
There is $v_1\in S_1\cap T_1$.
\item[{\rm (2)}]
There are $v_2,u_2\in S_2\cap T_2$ such that
$d_{G_2}(u_2,v_2)\ge d_{G_2}(u_2,z)$ for every $z\in V(G_2)$.
\end{itemize}
Let $G$ be a graph obtained by adding the edge $v_1v_2$ to the disjoint union
of the graphs $G_1$ and $G_2$.
Then, $\dim(G)=\dim(G_1)+\dim(G_2)-2$ and $\edim(G)=\edim(G_1)+\edim(G_2)-2$.
Moreover, $S=S_1\cup S_2-\{v_1,v_2\}$ is a metric basis of $G$
and $T=T_1\cup T_2-\{v_1,v_2\}$ is an edge metric basis of $G$.
\end{lemma}

\begin{proof}
First observe that for every $z_2\in V(G_2)$, the set
$\big(S_1-\{v_1\}\big)\cup\{z_2\}$ identifies the vertices of $G_1$.
Analogously, if $z_1\in V(G_1)$, then the set
$\big(S_2-\{v_2\}\big)\cup\{z_1\}$ identifies the vertices of $G_2$.
Since $G_1$ and $G_2$ are not paths, $|S_1|\ge 2$ and $|S_2|\ge 2$.
Hence, the set $S=S_1\cup S_2-\{v_1,v_2\}$ identifies the vertices of $G_1$ and it also identifies
the vertices of $G_2$.
Thus, to conclude that $S$ is a metric generator, we just may need to consider a pair of vertices $x,y$ with $x\in V(G_1)$ and $y\in V(G_2)$.

By (2), $d_{G_2}(u_2,v_2)\ge d_{G_2}(u_2,z)$ for every $z\in V(G_2)$.
Hence, for $y\in V(G_2)$, we have $d_G(u_2,v_2)\ge d_G(u_2,y)$,
while for $x\in V(G_1)$ it follows $d_G(u_2,x) \ge d_G(u_2,v_2)+1$. Thus, clearly $d_G(u_2,x)>d_G(u_2,y)$.
Since $u_2\in S$, we conclude that $S$ identifies all the vertices of $G$, and so it
is a metric generator for $G$.

Now suppose that $S'$ is a metric generator for $G$ such that $|S'|<|S|$.
Then either $|S'\cap V(G_1)|<|S_1|-1$ or $|S'\cap V(G_2)|<|S_2|-1$.
In the first case $(S'\cap V(G_1))\cup\{v_1\}$ identifies $G_1$ which
contradicts $\dim(G_1)=|S_1|$, while in the second case
$(S'\cap V(G_2))\cup\{v_2\}$ identifies $G_2$ which contradicts
$\dim(G_2)=|S_2|$.
Consequently, such a set $S'$ does not exist, and we obtain that $S$ is a metric basis for $G$, which means
$\dim(G)=\dim(G_1)+\dim(G_2)-2$.

The situation for the edge metric basis is analogous.
The only difference comes by noticing that we may need to consider the edge metric representation of the edge
$v_1v_2$, that is $r(v_1v_2|S)$, but this is unique in $G$.
Observe that if $z_1\in T_1-\{v_1\}$ and $z_2\in T_2-\{v_2\}$,
then $d_G(z_1,v_1v_2)<d_G(z_1,e_2)$ and $d_G(z_2,v_1v_2)<d_G(z_2,e_1)$ for every
$e_1\in E(G_1)$ and $e_2\in E(G_2)$.
\end{proof}


\subsection{Conclusion of the proof}

In order to complete the proof of Theorem \ref{th:main}, we make the following
construction, that uses Lemma \ref{lem:operation}.

For some positive integer $\ell\ge 2$, we consider $\ell$ graphs given as follows.
Let $G_1=G_{n_1,n_2,n_3}$, and let $G_2=G_3=\dots=G_{\ell}=G_{n_1,1,2}$,
with $n_1\ge 5$, $n_2\ge 1$ and $n_3\ge 2$.
To distinguish vertices in distinct copies of $G_i$, if $x$ is a vertex in $G_k$,
$1\le k\le\ell$, then we denote it by $x^k$.
Let $\alpha=\lfloor\frac{n_1+1}2\rfloor$.
Then $L^{\ell}_{n_1,n_2,n_3}$ is a graph obtained from the disjoint union
$G_1\cup G_2\cup\dots\cup G_{\ell}$ by adding the edges $a^1_{\alpha}j^2_1$,
$a^2_{\alpha}j^3_1$, \dots, $a^{\ell-1}_{\alpha}j^{\ell}_1$. See Figure \ref{fig-L-ell} for an example.

\begin{figure}[ht]
\centering
\includegraphics[width=140mm, height=30mm]{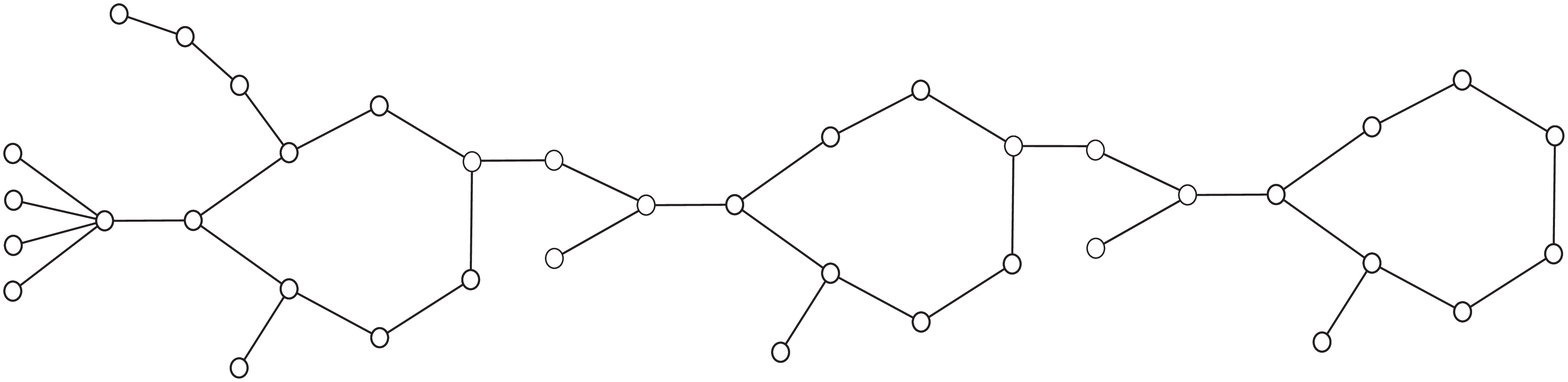}
\caption{The graph $L^3_{7,3,4}$. Note that $G_1$ is the graph of Figure \ref{fig:G-7-3-2}.}
\label{fig-L-ell}
\end{figure}

Obviously, $L^{\ell}_{n_1,n_2,n_3}$ is a connected graph, and if $\ell=1$, then $L^{\ell}_{n_1,n_2,n_3}$ is just $G_{n_1,n_2,n_3}$.
We next give some results concerning the metric and edge metric dimensions of such graphs.

\begin{lemma}
\label{thm:L}
Let $n_1\ge 5$, $n_2\ge 1$, $n_3\ge 2$ and $\ell\ge 1$.
Then the following holds.
\begin{itemize}
\item[{\rm (1)}]
If $n_1$ is odd, then $\dim(L^{\ell}_{n_1,n_2,n_3})=n_3$ and
$\edim(L^{\ell}_{n_1,n_2,n_3})=n_3+\ell$.
\item[{\rm (2)}]
If $n_1$ is even, then $\dim(L^{\ell}_{n_1,n_2,n_3})=n_3+\ell$ and
$\edim(L^{\ell}_{n_1,n_2,n_3})=n_3$.
\end{itemize}
\end{lemma}

\begin{proof}
We only prove the result for the case when $n_1$ is odd, since the proof for the case
when $n_1$ is even is in fact the same.

Let $\beta=\lceil\frac{n_1+3}2\rceil$.
As shown in the proof of Lemma~{\ref{prop:vertex}},
$\dim(G_{n_1,n_2,n_3})=n_3$ and $S=\{j_1,j_2,\dots,j_{n_3-1},a_{\alpha}\}$ is a
metric basis of $G_{n_1,n_2,n_3}$.
By Lemma~{\ref{prop:edge}}, $\edim(G_{n_1,n_2,n_3})=n_3+1$ and
$T=\{j_1,j_2,\dots,j_{n_3-1},a_{\alpha},a_{\beta}\}$ is an edge
metric basis of $G_{n_1,n_2,n_3}$. We first consider the graphs $G_1$ and $G_2$, in concordance with the construction of $L^{\ell}_{n_1,n_2,n_3}$. Now, for $i\in \{1,2\}$, denote the metric basis $S$ and an edge metric basis $T$ in $G_i$ by $S_i$ and $T_i$,
respectively.
Then $a_{\alpha}^1\in S_1\cap T_1$ and $j_1^2,a_{\alpha}^2\in S_2\cap T_2$,
and moreover, $d_{G_2}(a_{\alpha}^2,j_1)\ge d_{G_2}(a_{\alpha}^2,z)$ for every
$z\in V(G_2)$.
Hence, the graphs $G_1$ and $G_2$ satisfy the assumptions of
Lemma~{\ref{lem:operation}}.
Thus, $L^2_{n_1,n_2,n_3}$ has metric dimension $n_3$ and edge metric
dimension $n_3+2$.
Moreover, $S=S_1\cup S_2-\{a_{\alpha}^1,j_1^2\}$ is a metric basis of
$L^2_{n_1,n_2,n_3}$ and $T=T_1\cup T_2-\{a_{\alpha}^1,j_1^2\}$ is an edge
metric basis of $L^2_{n_1,n_2,n_3}$, for which $a_{\alpha}^2\in S\cap T$.

Since $L^{2}_{n_1,n_2,n_3}$ and $G_3$ satisfy the assumptions of
Lemma~{\ref{lem:operation}}, we can then proceed with $L^{2}_{n_1,n_2,n_3}$ and $G_3$ instead of
$G_1$ and $G_2$, and continue with this process until we reach the largest value of $\ell$.
This concludes the proof.
\end{proof}

By using the exposition of results above, we are then able to complete the proof of Theorem \ref{th:main}. That is, if $k_1<k_2$, then $\dim(L^{k_2-k_1}_{5,n_2,k_1})=k_1$ and $\edim(L^{k_2-k_1}_{5,n_2,k_1})=k_2$, by Lemma~{\ref{thm:L}}.
Hence, if $n_0=|V(L^{k_2-k_1}_{5,1,k_1})|$, then for every $n\ge n_0$ the
graph $L^{k_2-k_1}_{5,1+n-n_0,k_1}$ has the required properties.

On the other hand, if $k_1>k_2$, then $\dim(L^{k_1-k_2}_{6,n_2,k_1})=k_1$ and
$\edim(L^{k_1-k_2}_{6,n_2,k_2})=k_2$, by Lemma~{\ref{thm:L}}.
Hence, if $n_0=|V(L^{k_1-k_2}_{6,1,k_2})|$, then for every $n\ge n_0$ the
graph $L^{k_1-k_2}_{6,1+n-n_0,k_2}$ has the required properties. \hfill $\blacksquare$ \medskip

\section{Further work}

The graphs from Figure \ref{fig:G-7-3-2}, together with the graphs $L^{\ell}_{n_1,n_2,n_3}$ defined above, for $n_1$ even, allow to think that characterizing the whole class of graphs $G$ for which $\edim(G)<\dim(G)$ is a highly challenging problem, since the structures that such graphs can have is rather wide. In this concern, observe also for instance the graphs of Figure \ref{fig:g10}, which have order 11, and other examples are the already mentioned torus graphs $C_{4r}\square C_{4t}$.

\begin{figure}[ht]
\centering
\includegraphics[width=70mm, height=60mm]{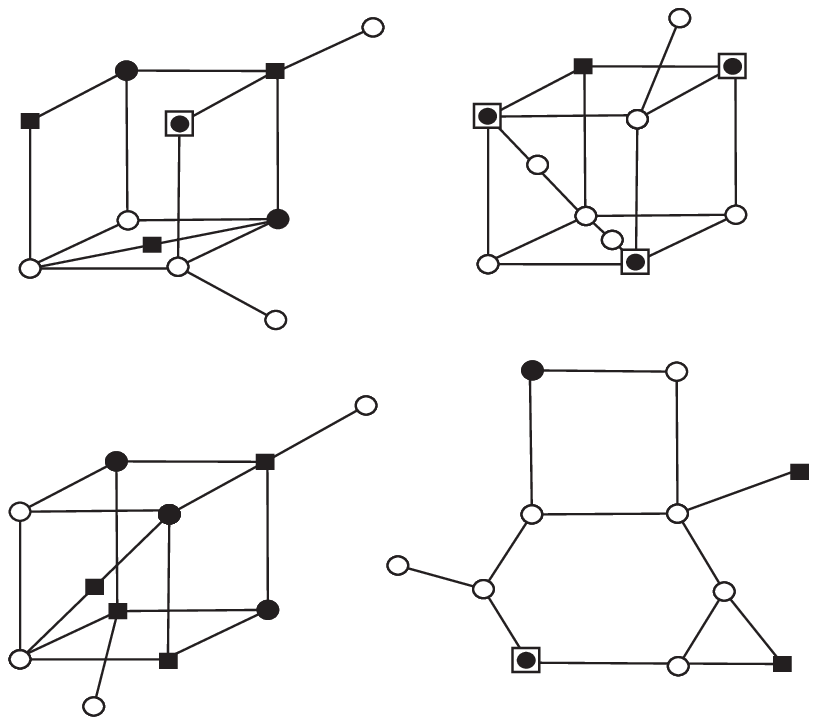}
\caption{Some graphs $G$ with $11$ vertices for which  $\dim(G)>\edim(G)$. As in Figure \ref{fig:g10}, the squared vertices form a metric basis and the circled bolded vertices form an edge metric basis.}
\label{fig:g10}
\end{figure}

Notwithstanding, one could think into characterizing some special families of graphs achieving this property. Thus, some open problems that would be of interest from our point of view are the following ones.
\begin{itemize}
  \item Characterize the class of unicyclic graphs $G$ for which $\edim(G)<\dim(G)$.
  \item Characterize all the graphs (or maybe only the unicyclic ones) $G$ for which $\edim(G)=\dim(G)-1$.
  \item Characterize all the graphs $G$ for which ($\edim(G)=2$ and $\dim(G)=3$) or ($\edim(G)=3$ and $\dim(G)=4$).
  \item Find some necessary and/or sufficient conditions for a connected graph $G$ to satisfy that $\edim(G)<\dim(G)$.
\end{itemize}

\vskip 1pc
\noindent{\bf Acknowledgements.}~~The authors acknowledge partial support
by Slovak research grants VEGA 1/0142/17, VEGA 1/0238/19, APVV--15--0220,
APVV--17--0428, Slovenian research agency ARRS program \ P1--0383  and project J1-1692.


\begin{thebibliography}{99}

\bibitem{Blumenthal1953}
L.~M. Blumenthal, \textit{Theory and applications of distance geometry}, Oxford
  University Press, Oxford (1953).

\bibitem{Caceres} J. Caceres, C. Hernando, M. Mora, I. M. Pelayo, M. L. Puertas, C. Seara, and D. R. Wood, On the metric dimension of Cartesian products of graphs, \emph{SIAM J. Discrete Math.} \textbf{21} (2) (2007) 423--441.

\bibitem{Filipovi2019} V. Filipovi\'c, A. Kartelj, and J. Kratica, Edge metric dimension of some generalized Petersen graphs, \emph{Results Math.} 74 (4) (2019) article \# 182.

\bibitem{Geneson2020} J. Geneson, Metric dimension and pattern avoidance in graphs, \emph{Discrete Appl. Math.} (2020). In press. DOI: 10.1016/j.dam.2020.03.001

\bibitem{Harary1976}
F. Harary and R.~A. Melter, On the metric dimension of a graph, \emph{Ars Combin.} \textbf{2} (1976) 191--195.

\bibitem{Kel} A. Kelenc, N. Tratnik, and I. G. Yero, Uniquely identifying the edges of a graph: the edge metric dimension,
\emph{Discrete Appl. Math.} {\bf 251} (2018) 204--220.

\bibitem{Liu} J. B. Liu, Z. Zahid, R. Nasir, and W. Nazeer, Edge version of metric dimension and doubly resolving sets of the necklace graph, \emph{Mathematics} \textbf{6} (11) (2018) article \# 243.

\bibitem{Nasir} R. Nasir, S. Zafar, and Z. Zahid, Edge metric dimension of graphs, \emph{Ars Combin.} In press.

\bibitem{Peterin2020} I. Peterin and I. G. Yero, Edge metric dimension of some graph operations, \emph{Bull. Malays. Math. Sci. Soc.} \textbf{43} (2020) 2465--2477.

\bibitem{Sebo2004} A. Seb\H{o} and E. Tannier, On metric generators of graphs, \emph{Math. Oper. Res.} \textbf{29} (2) (2004) 383--393.

\bibitem{Slater1975}
P.~J. Slater, Leaves of trees, \emph{Congr. Numer.} \textbf{14} (1975) 549--559.

\bibitem{Zhang2020} Y. Zhang and S. Gao, On the edge metric dimension of convex polytopes and its related graphs, \emph{J. Comb. Optim.} \textbf{39} (2) (2020) 334-350.

\bibitem{Zhu} E. Zhu, A. Taranenko, Z. Shao, and J. Xu, On graphs with the maximum edge metric dimension, \emph{Discrete Appl. Math.} \textbf{257} (2019) 317--324.

\bibitem{Zubrilina} N. Zubrilina, On the edge dimension of a graph, \emph{Discrete Math.} \textbf{341} (7) (2018) 2083--2088.

\end{thebibliography}
\end{document}